\documentclass[11pt]{amsart}

\makeatletter
\@namedef{subjclassname@2020}{%
  \textup{2020} Mathematics Subject Classification}
\makeatother

\newtheorem{theorem}{Theorem}[section] 
\newtheorem{lemma}[theorem]{Lemma}
\newtheorem{proposition}[theorem]{Proposition}

\newtheorem{corollary}[theorem]{Corollary}
\theoremstyle{definition}
\newtheorem{definition}[theorem]{Definition}

\theoremstyle{remark}

\numberwithin{equation}{section}

\DeclareMathOperator{\Lip}{Lip}
\DeclareMathOperator{\dist}{dist}

\DeclareMathOperator{\id}{id}

\title[Growth rate of Lipschitz constants]{Growth rate of Lipschitz constants for retractions between finite subset spaces}

\author{Earnest Akofor}
\address{215 Carnegie, Mathematics Department, Syracuse University, Syracuse, NY 13244, USA}
\email{eakofor@syr.edu}

\author{Leonid V. Kovalev}
\address{215 Carnegie, Mathematics Department, Syracuse University, Syracuse, NY 13244, USA}
\email{lvkovale@syr.edu}
\thanks{L.V.K. supported by the National Science Foundation grant DMS-1764266.}



\begin{document}
\baselineskip5.7mm

\maketitle

\begin{abstract}
For any metric space $X$, finite subset spaces of $X$ provide a sequence of isometric embeddings $X=X(1)\subset X(2)\subset\cdots$. The existence of Lipschitz retractions $r_n\colon X(n)\to X(n-1)$ depends on the geometry of $X$ in a subtle way. Such retractions are known to exist when $X$ is an Hadamard space or a finite-dimensional normed space. But even in these cases it was unknown whether the sequence $\{r_n\}$ can be uniformly Lipschitz. We give a negative answer by proving that $\Lip(r_n)$ must grow with $n$ when $X$ is a normed space or an Hadamard space.
\end{abstract}

\section{Introduction}
 
Given a topological space $X$ and a positive integer $n$, 
the nonempty subsets of $X$ of cardinality at most $n$ form another topological space $X(n)$ with a natural quotient topology induced by the map that takes each ordered tuple $(x_1,\dots,x_n)$ in the Cartesian product $X^n$ to the finite set $\{x_1,\dots, x_n\}$ in $X(n)$. The space $X(n)$ is called the \emph{$n$th finite subset space} of $X$ (the terms symmetric product or symmetric power are sometimes used as well). This concept goes back to Borsuk and Ulam~\cite{BorsukUlam}. When $X$ is a metric space, $X(n)$ becomes a metric space with respect to \emph{Hausdorff distance} which is  given by
\begin{equation}\label{def-dh}
\begin{split}
& d_H\big(\{x_1,\dots,x_n\}, \{x_1',\dots,x_n'\}\big) \\ 
& := \max\left\{\max_i\min_j d(x_i,x_j'), \max_i\min_j d(x_j,x_i') \right\}.
\end{split}
\end{equation}
See, e.g.,~\cite[Proposition 1.2.2]{Akoforthesis} for details of this metrization.
The natural embeddings $X=X(1)\subset X(2)\subset\cdots$ are isometric with respect to the Hausdorff distance.  

If $X$ and $Y$ are metric spaces, a map $f\colon X\to Y$ is called \emph{Lipschitz} if there is a number $L\geq 0$ such that $d(f(x),f(x'))\leq L d(x,x')$ for all $x,x'\in X$. The least of such numbers $L$ is denoted by $\Lip(f)$ and is called the \emph{Lipschitz constant} of $f$.  If $Y\subset X$, a map $r\colon X\to Y$ is a \emph{retraction} if its restriction to $Y$ is the identity map. If $r$ is in addition Lipschitz, it is called a \emph{Lipschitz retraction}. 
 
A Lipschitz retraction $X(n)\to X(k)$, where $k<n$, can be interpreted as a robust choice of $k$ clusters within a finite set. Indeed, given a set $A\subset X$ of cardinality $|A|\le n$, we must choose a set $r(A)$ with $|r(A)|\le k$ subject to the conditions that $r(A)=A$ if $|A|\le k$, and $r(A)$ a Lipschitz function of $A$.  

Some spaces $X$ present topological obstructions to the existence of such retractions. For example, if $X$ is the circle $S^1$, then $X(3)$ is homeomorphic to $3$-sphere~\cite{Bott} and cannot be retracted onto $X(1)=S^1$, being simply-connected. Section~\ref{sec:circle} presents a more general obstruction of this type.  
If $X$ is a Hilbert space of any dimension (finite or infinite), for every $n$ there exists a Lipschitz retraction $r_n\colon X(n)\to X(n-1)$ with $\Lip(r_n) \le \max\left(n^{3/2}, 2n-1\right)$ \cite[(2.5)]{Kovalev2016}. Question~3.2 in~\cite{Kovalev2016} and Remark~3.5 in~\cite{BacakKovalev} asked whether $\Lip(r_n)$ can be bounded independently of $n$. Our first result shows that it must grow at least linearly with respect to $n$, provided that $\dim X\ge 2$. 

\begin{theorem}\label{thm:normed}
Let $X$ be a normed space over $\mathbb R$ with $\dim X\ge 2$. Suppose that $r\colon X(n)\to X(k)$ is a Lipschitz retraction, where $1\le k\le n-1$. Then 
\begin{equation}\label{normed1}
\Lip(r)\geq \frac{kn}{2\pi(n-1)} - \frac{1}{2}.   
\end{equation} 
Moreover, if $X$ is a Hilbert space, then 
\begin{equation}\label{normed2}
\Lip(r)\geq \frac{kn}{\pi(n-1)} - 1.   
\end{equation} 
\end{theorem}

The case $\dim X=1$, when $X$ is isometric to $\mathbb R$, is covered by our second main theorem. It concerns \emph{Hadamard spaces}, which share the geometric properties of Hilbert spaces but not necessarily their linear structure. To define them, we need the notion of a geodesic space. 

A \emph{geodesic} in a metric space $(X, d)$ is a mapping 
$\gamma\colon [0, 1]\to X$ such that for all $t, s\in [0, 1]$ we have $d(\gamma(t), \gamma(s)) = |t-s| d(\gamma(0), \gamma(1))$. In geometric terms, a geodesic is a curve parametrized proportionally to its arclength. If for any two points $p, q\in X$ there exists a geodesic with $\gamma(0)=p$ and $\gamma(1)=q$, then $X$ is called a \emph{geodesic space}. 

\begin{definition}\label{def:Hadamard}
A complete geodesic space is called an \emph{Hadamard space} if for every point $z$ and every geodesic $\gamma$ we have 
\[
d(\gamma(t), z)^2 
\le (1 - t)d (\gamma(0), z)^2 + t d (\gamma(1), z)^2
- t(1 - t) d (\gamma(0), \gamma(1))^2
\]
for all $t\in [0, 1]$.
\end{definition}

We refer to~\cite{Bacak} for equivalent definitions and the motivation behind the concept of an  Hadamard space. Theorem~3.2 in~\cite{BacakKovalev} asserts that for every Hadamard space $X$ and every $n\ge 2$ there exists a Lipschitz retraction $r_n\colon X(n)\to X(n-1)$ with $\Lip(r_n)\le \max\left(2n^2+\sqrt{n}, 4n^{3/2}+1\right)$. The following theorem gives a lower bound for $\Lip(r_n)$, answering a question posed in~\cite[Remark 3.5]{BacakKovalev}. 

\begin{theorem}\label{thm:metric}
Let $X$ be either a normed space over $\mathbb R$ of dimension $\dim X\ge 1$, or 
an Hadamard space  containing more than one point. If $r\colon X(n)\to X(n-1)$ is a Lipschitz retraction, then $\Lip(r)\geq n-3$. 
\end{theorem}

The paper is organized as follows. In Section~\ref{sec:circle} we collect the necessary results from the algebraic topology of finite subset spaces. Section~\ref{sec:displacement} contains preliminary results on the properties of Lipschitz retractions. Theorems~\ref{thm:normed} and~\ref{thm:metric} are proved in sections~\ref{sec:normed} and~\ref{sec:metric}, respectively. Corollary~\ref{cor:metric} provides a more general version of Theorem~\ref{thm:metric}. Note that Theorem~\ref{thm:metric} gives a slightly better lower bound than Theorem~\ref{thm:normed}. On the other hand, Theorem~\ref{thm:normed} applies to retractions onto $X(k)$ for any $k<n$, not only $k=n-1$.

\section{Topology of finite subset spaces}\label{sec:circle}

Let $H_n(X)$ denote the $n$th singular homology group of a topological space $X$~\cite[p. 108]{Hatcher}. The homology groups of finite subset spaces $S^1(n)$ were computed by Wu~\cite{Wu} and their homotopy type was determined by Tuffley~\cite{Tuffley}, see also~\cite{ChinenKoyama}. 

\begin{theorem}\label{HomologyFSS1} (\cite[Theorem III]{Wu}, \cite[Theorems 4--5]{Tuffley}). Given $n\in \mathbb N$, let $m$ be the largest odd integer not exceeding $n$. The homology groups $H_0(S^1(n))$ and $H_m(S^1(n))$ are isomorphic to $\mathbb Z$, and all other homology groups of $S^1(n)$ are trivial. Moreover, when $n$ is odd, the inclusion $S^1(n)\subset S^1(n+1)$ induces the doubling map $k\mapsto 2k$ on the homology group $H_{n}(S^1(n))$.
\end{theorem}

The homology presents an obstruction to the existence of continuous retractions between the finite subset spaces of the circle. 

\begin{proposition}\label{NoContRetI}
If $1\leq k\leq n-1$, there is no continuous retraction $S^1(n)\to S^1(k)$.
\end{proposition}

\begin{proof} Suppose there exists a continuous retraction $r\colon S^1(n)\to S^1(k)$. The map induced by $r$ on the homology groups of $S^1(n)$ is a left inverse of the map induced by the inclusion of $S^1(k)$ into $S^1(n)$~\cite[p. 111]{Hatcher}.  In particular, the latter map is injective.

Let $m$ be the greatest odd integer not exceeding $k$. Since $H_m(S^1(k))$ is isomorphic to $\mathbb Z$, the group $H_m(S^1(n))$ must be nontrivial. Theorem~\ref{HomologyFSS1} implies that $m$ is the greatest odd integer not exceeding $n$. This is only possible if $n$ is even and $k = m = n-1$. However, in this case the inclusion of $S^1(n-1)$ into $S^1(n)$ induces the doubling map on the homology groups $H_{n-1}$, and this map does not have a left inverse, contradicting the previous paragraph. 
\end{proof}

The obstruction presented by Proposition~\ref{NoContRetI} is the basis of our proof of Theorem~\ref{thm:normed}. To prove Theorem~\ref{thm:metric} we need to find some topological obstruction within the subsets of $\mathbb R$. It is provided by \textit{pinned finite subset spaces}, which are defined as follows.

\begin{definition} Given a set $U\in X(n)$, the \emph{$U$-pinned finite subset space} of $X$ is
\[
X(n, U) := \{A\in X(n) \colon U\subset A \}.
\]
\end{definition}

The space $X(n, U)$ is empty when $n<|U|$. Notable examples of pinned finite subset spaces include  
\begin{equation}\label{def:dunce}
 \mathcal D^n := I(n+2, \{0, 1\})   
\end{equation}
where $I$ is the interval $[0, 1]$ and $n=0, 1, \dots$. The studies of $\mathcal D^n$ go back to Schori~\cite{Schori} who proved that $I(n)$ is a double cone of $\mathcal D^{n-2}$ when $n\ge 2$. Andersen, Marjanovi\'{c} and Schori~\cite{AMS} called the spaces $\mathcal D^n$ with even $n$ ``higher-dimensional dunce hats'' because  $\mathcal D^2$ is homeomorphic to the ``dunce hat'' space introduced by Zeeman~\cite{Zeeman} and $\mathcal D^n$  shares some features of $\mathcal D^2$ when $n$ is even.  

\begin{theorem}\label{dunce} \cite[Theorem 3.4]{AMS}. When $n$ is even, $\mathcal D^n$ is contractible. When $n$ is odd, $\mathcal D^n$ has the homotopy type of $S^n$.
\end{theorem}

As with the ordinary finite subset spaces, we have natural inclusions $X(k, U)\subset X(n, U)$ when $|U|\le k\le n$. For example, $\mathcal D^0\subset \mathcal D^1\subset \cdots$. 

\begin{corollary}\label{NoContRetII}
When $n$ is even, there is no continuous retraction of $\mathcal D^n$ onto $\mathcal D^{n-1}$.
\end{corollary}

\begin{proof} By Theorem~\ref{dunce}, the homology group $H_{n-1}(\mathcal D^{n-1})$ is isomorphic to $\mathbb Z$ while $H_{n-1}(\mathcal D^{n})$ is trivial. Since the former group does not embed in the latter, $\mathcal D^{n-1}$ is not a retract of $\mathcal D^{n}$.
\end{proof}

Corollary~\ref{NoContRetII} is a less complete result than Proposition~\ref{NoContRetI}. We do not know if $\mathcal D^n$ retracts onto $\mathcal D^{n-1}$ when $n$ is odd. An interesting related question is whether each higher-dimensional dunce hat $\mathcal D^{2k}$ is an \emph{absolute Lipschitz retract}, meaning that it is a Lipschitz retract in any metric space that contains it. At present it is not known whether finite subset spaces inherit the absolute Lipschitz retract property: see~\cite{Akofor, Akoforthesis, Kovalev2015} for partial results. 

Another example of a pinned finite subset space is $S^1(n, \{1\})$ which was studied by Tuffley~\cite[p. 1131]{Tuffley}. This space is homeomorphic to $\mathcal D^{n-1}$. More specificially, the map $A\mapsto \{e^{2\pi it}\colon t\in A\}$ is a bi-Lipschitz homeomorphism of $\mathcal D^{n-1}$ onto $S^1(n, \{1\})$: see the proof of Theorem~3.1 in~\cite{BIY}.   

\section{Properties of Lipschitz retractions} \label{sec:displacement}

Let $X$ be a metric space and $n\ge 2$. The \emph{minimum separation} function $\delta_n\colon X(n)\to [0, \infty)$ is defined as follows: 
\begin{equation}\label{minsep}
\delta_n(A) = \begin{cases}
\min\{d_X(p, q)\colon p, q\in A, p\ne q\}  &\text{if  } |A|=n \\ 
0 & \text{if } |A|<n
\end{cases}
\end{equation}

The importance of $\delta_n$ stems from the following observation. 

\begin{lemma}\label{distance}
Suppose $X$ is a metric space and $n\ge 2$. For each $A\in X(n)$ there is $B\in X(n-1)$ such that $d_H(A, B)\le \delta_n(A)$. If, in addition, $X$ is a geodesic space, then the conclusion can be strengthened to $d_H(A, B)\le \frac12\delta_n(A)$.
\end{lemma}

\begin{proof} If $|A|<n$, the set $B=A$ satisfies the conclusion. Suppose $|A|=n$. Let $p, q\in A$ be two points such that $d_X(p, q) = \delta_n(A)$. Then the set $B=A\setminus \{p\}$ has $n-1$ elements and $d_H(A, B) \le \delta_n(A)$. 

If $X$ is a geodesic space, let $m$ be the midpoint of a geodesic from $p$ to $q$ and define $
B=(A\cup \{m\})\setminus \{p,q\} $. 
This set has $n-1$ elements and $d_H(A, B) \le \frac{1}{2}\delta_n(A)$.
\end{proof}

\begin{lemma}\label{displacement} Let $X$ be a metric space and $1\le k<n$. Suppose $r\colon X(n)\to X(k)$ is a Lipschitz retraction. Then for every $A\in X(n)$ we have
\begin{equation}\label{disp1}
d_H(r(A), A) \le (\Lip(r)+1) \dist_H(A, X(k))
\end{equation}
where $\dist_H(A, X(k)) = \inf \{d_H(A, B)\colon B\in X(k)\}$. 
\end{lemma}

\begin{proof} For every $B\in X(k)$ we have $r(B)=B$, hence
\[
d_H(r(A), B) = d_H(r(A), r(B)) \le \Lip(r) d_H(A, B).
\]
By the triangle inequality, $d_H(r(A), A)\le (\Lip(r)+1)d_H(A, B)$. Taking the infimum over $B\in X(k)$ yields~\eqref{disp1}.
\end{proof}

\begin{lemma}\label{retracts-inherit} Let $X$ be a metric space such that there exists a Lipschitz retraction $r\colon X(n)\to X(k)$ for some integers $1\le k < n$. Suppose $Y$ is a metric space such that there exist Lipschitz maps $f\colon X\to Y$ and $g\colon Y\to X$ with the property $f\circ g = \id_Y$. Then there exists a Lipschitz retraction $s \colon Y(n)\to Y(k) $ with $\Lip(s)\le \Lip(f)\Lip(g)\Lip(r)$.
\end{lemma}

\begin{proof} The map $g$ induces a map $g_n \colon Y(n)\to X(n)$ such that $g_n(A)$ is the image of set $A$ under $g$. From the definition~\eqref{def-dh} of Hausdorff distance it is easy to see that $\Lip(g_n)=\Lip(g)$. Similarly, $f$ induces a map $f_{k}\colon X(k)\to Y(k)$. Let $s = f_{k}\circ r\circ g_n$. By construction, $s$ maps $Y(n)$ to $Y(k)$ and has Lipschitz constant at most $\Lip(f)\Lip(g)\Lip(r)$. If $A\in Y(k)$ then $g_n(A)\in X(k)$, hence $s(A) = f_{k}(g_n(A)) = A$ by the property $f\circ g = \id_Y$. Thus $s$ is a Lipschitz retraction onto $Y(k)$.
\end{proof}

Two useful special cases of Lemma~\ref{retracts-inherit} are: (a) $Y$ is a Lipschitz retract of $X$, with $f$ being the inclusion map; (b) $Y$ is bi-Lipschitz equivalent to $X$, with $g=f^{-1}$.  

\section{Normed spaces: proof of theorem~\ref{thm:normed}} \label{sec:normed}

Let $X$ be a normed space over $\mathbb R$ of dimension at least $2$. The following statement is a special case of~\cite[Proposition G.1]{BenyaminiLindenstrauss} which summarizes the results of John~\cite{John} and Kadets-Snobar~\cite{KS}. 

\begin{lemma}\label{lem:plane} Let $Z$ be a $2$-dimensional subspace of a normed space $X$. Then there exists a linear projection  $P\colon X\to Z$, and a linear isomorphism $T\colon  \mathbb R^2\to Z$ such that $\|P\|\le \sqrt{2}$, $\|T\|\le \sqrt{2}$, and $\|T^{-1}\|\le 1$. 
\end{lemma}

Lemma~\ref{lem:plane} leads us to consider the geometry of finite subsets of $\mathbb R^2$ which is the subject of the following lemma.  

\begin{lemma}\label{lem:distancecircle}
Let $S^1\subset \mathbb R^2$ be the unit circle centered at $0$, equipped with the arclength metric. For any set $A\in S^1(n)$ and any $k\in 1, \dots, n-1$ there exists $B\in S^1(k)$ such that $d_H(A, B)\le \pi(n-1)/(kn)$. 
\end{lemma}

\begin{proof} Let $A_j$ be the result of rotating $A$ by the angle $2\pi j/k$, and let $R = \bigcup_{j=1}^k A_j$. Since $R$ has at most $kn$ points, its complement in $S^1$ contains an open arc of length $2\pi/(kn)$. The $k$-fold symmetry of $R$ implies that it is covered by $k$ uniformly spaced closed arcs $\Gamma_1,\dots, \Gamma_k\subset S^1$ of length 
\[
\frac{2\pi}{k} - \frac{2\pi}{kn}
= \frac{2\pi(n-1)}{kn}.
\]
Therefore, $A\subset \bigcup_{j=1}^k \Gamma_j$.
Let $B$ be the set of midpoints of all arcs $\Gamma_j$ such that $\Gamma_j\cap A$ is nonempty. Then every point of $B$ is within distance at most $\pi (n-1)/(kn)$ of some point of $A$, and vice versa. Since $|B|\le k$, the lemma is proved.
\end{proof}

The estimate in Lemma~\ref{lem:distancecircle} is best possible when $k=n-1$, as one can check using a set $A$ of $n$ equally spaced points. Lemma~\ref{lem:distancecircle} also applies when $S^1$ is equipped with chordal metric, i.e. the restriction of the Euclidean metric on $\mathbb R^2$, because the chordal metric is majorized by arclength. 

\begin{proof}[Proof of Theorem~\ref{thm:normed}] 
Let $Y=\mathbb R^2$ and let $Z$, $P$, $T$ be as in Lemma~\ref{lem:plane}. The mappings $f = T^{-1}\circ P$ and $g = T$ satisfy the assumptions of Lemma~\ref{retracts-inherit}. Therefore, there exists a retraction $s\colon Y(n)\to Y(k)$ with $\Lip(s)\le 2\Lip(r)$. When $X$ is a Hilbert space, this can be improved to $\Lip(s) \le \Lip(r)$ because any two-dimensional subspace $Z\subset X$ is isometric to $\mathbb R^2$ and $P$ can be the orthogonal projection onto $Z$.  Therefore, both cases of Theorem~\ref{thm:normed} will be proved if we show that $\Lip(s) \ge \frac{kn}{\pi(n-1)} - 1$. Suppose, toward a contradiction, that the constant $c := \frac{\pi(n-1)}{kn}(\Lip (s) + 1)$  satisfies $c<1$. 

Let $S^1\subset \mathbb R^2$ be the unit circle centered at $0$. From  Lemmas~\ref{displacement} and~\ref{lem:distancecircle} it follows that $d_H(s(A), A) \le c$ for every $A\in S^1(n)$. Hence $s(A)$
is contained in the set $W = \{x\in \mathbb R^2\colon \|x\|\ge 1-c\}$. The radial projection $f(x) = x/\|x\|$ provides a Lipschitz retraction of $W$ onto $S^1$. Therefore, the mapping $A\mapsto f(s(A))$ is a Lipschitz retraction on $S^{1}(n)$ onto $S^1(k)$. This contradicts Proposition~\ref{NoContRetI}. 
\end{proof}

\section{Metric spaces: proof of theorem~\ref{thm:metric}} \label{sec:metric}

We begin with the special case of retractions between the finite subset spaces of an interval on the real line. The definition of the Hausdorff distance~\eqref{def-dh} implies that for any two nonempty finite sets $A, B\subset \mathbb R$ we have 
\begin{equation}\label{maxmin}
|\max A - \max B| \le d_H(A, B),\qquad 
|\min A - \min B| \le d_H(A, B).
\end{equation} 

\begin{theorem}\label{thm:interval} Let $I = [a,b]\subset \mathbb R$ where $-\infty<a<b<\infty$. If $r\colon I(n)\to I(n-1)$ is a Lipschitz retraction, then 
\begin{equation}\label{interval1}
    \Lip(r) \ge \begin{cases} n-2, & \text{$n$ is even} \\ 
    n-3, & \text{$n$ is odd}
\end{cases}     
\end{equation}
\end{theorem}

\begin{proof} The choice of the interval $[a, b]$ does not matter in this theorem. Indeed, if $\phi\colon \mathbb R\to\mathbb R$ is an invertible affine transformation, then $A\mapsto \phi(r(\phi^{-1}(A)))$ is a Lipschitz retraction between the finite subset spaces of $\phi(I)$, with the same Lipschitz constant $\Lip(r)$. Thus, we can choose any convenient interval $I$ in the proof, and we use two different intervals for the two cases that follow. 

 \textbf{Case 1: $n$ is even.} Assume $n\ge 4$ since the statement is trivial for $n=2$. Let $I = [0, 1]$. 
Toward a contradiction, suppose that the quantity $c:= (\Lip(r)+1)/(n-1)$ satisfies $c<1$.  Recall the subsets $\mathcal D^{n-2} \subset I(n)$ defined by~\eqref{def:dunce}. For every set $A\in \mathcal D^{n-2}$ the minimal separation~\eqref{minsep} satisfies $\delta_n(A)\le 1/(n-1)$. From Lemmas~\ref{distance} and~\ref{displacement} it follows that
\[
d_H(r(A), A) \le 
\frac{1}{2(n-1)}(\Lip(r)+1) \le c/2.
\]
By~\eqref{maxmin} we have $\min r(A)\le c/2$ and $\max r(A) \ge 1- c/2$. For $t\in [0, 1]$ let 
\begin{equation}\label{def-fA}
f_A(t) = \frac{t - \min r(A)}{\max r(A) - \min r(A)}.
\end{equation}
Since the denominator in~\eqref{def-fA} is bounded below by $1-c$, the function $f_A$ is $(1-c)^{-1}$-Lipschitz with respect to $t$. It is also Lipschitz continuous with respect to $A$ by virtue of~\eqref{maxmin}. Therefore, the mapping $s(A) := f_A(r(A))$ is Lipschitz continuous on $\mathcal D^{n-2}$. By construction, the set $s(A)$ consists of at most $n-1$ points, $\min s(A) = 0$, and $\max s(A) = 1$. Thus $s(A)\in \mathcal D^{n-3}$. If $|A| < n$ then $r(A) = A$, hence $s(A)=A$. We have proved that $s$ is a Lipschitz retraction of $\mathcal D^{n-2}$ onto $\mathcal D^{n-3}$, which contradicts Corollary~\ref{NoContRetII}. 

\textbf{Case 2: $n$ is odd.} We may assume $n\ge 5$ since the statement is trivial for $n\le 3$.  Let $I=[0, 2]$.  Toward a contradiction, suppose that the quantity $c:= (\Lip(r)+1)/(n-2)$ satisfies $c<1$. Let $Y = [0, 1]\cup \{2\}$ and consider the pinned finite subset space $\mathcal E = Y(n, \{0, 1, 2\})$.  For every set $A\in \mathcal E$, all but one of its points lie in $[0, 1]$. Hence its minimal separation satisfies $\delta_n(A)\le 1/(n-2)$. By Lemmas~\ref{distance} and~\ref{displacement} we have
\[
d_H(r(A), A) \le 
\frac{1}{2(n-2)}(\Lip(r)+1) \le c/2.
\]
Since $\{0, 1, 2\}\subset A$, it follows that $r(A)$ meets each of the intervals $[0, c/2]$, $[1 - c/2, 1 +c/2]$, and $[2-c/2, 2]$. Moreover, $r(A)$ is disjoint from the interval $(1+c/2, 2-c/2)$. 

Let 
\[
s(A) = r(A)\cap [0, 1 + c/2] = r(A)\cap [0, 2-c/2). 
\]
Note that $|s(A)|\le r(A)-1 \le n-2$ since $r(A)$ meets $[2-c/2, 2]$. 
Suppose that $A, B\in \mathcal E$ are such that $\Delta: = d_H(r(A), r(B)) < 1 - c$. By the definition of $d_H$, for every $a\in s(A)$ there exists $b\in r(B)$ such that $|a-b|\le \Delta$. Then $b\le (1+c/2) + \Delta < 2-c/2$, which implies $b\in s(B)$. In conclusion, 
\begin{equation}\label{intersect-dH}
d_H(s(A), s(B)) \le d_H(r(A), r(B))
\end{equation}
whenever the right hand side is less than $1-c$. Since $r$ is Lipschitz continuous, ~\eqref{intersect-dH} shows that $s$ is also Lipschitz continuous. 

For $t\in [0, 1]$ let 
\[ 
f_A(t) = \frac{t - \min s(A)}{\max s(A) - \min s(A)}
\] 
and note that the denominator is bounded below by $1 - c$. As in  Case~1, it follows that $f_A(t)$ is Lipschitz with respect to both $t$ and $A$.  Therefore, the mapping $\sigma(A): = f_A(s(A))$ is Lipschitz continuous on $\mathcal E$. By construction, the set $\sigma(A)$ consists of at most $n-2$ points,
$\min \sigma(A) = 0$, and $\max\sigma(A) = 1$. Thus, $\sigma(A)\in  \mathcal D^{n-4}$. If $A\in \mathcal E$ has fewer than $n$ points, then $r(A) = A$, hence $\sigma(A)=A\cap [0, 1]$. 

The space $\mathcal D^{n-3}$ is isometric to $\mathcal E$ via the map $\iota(A) = A\cup \{2\}$. The previous paragraph shows that the composition $\sigma\circ \iota$ is a Lipschitz retraction of $\mathcal D^{n-3}$ onto $\mathcal D^{n-4}$. Since $n-3$ is even, we have a contradiction with   Corollary~\ref{NoContRetII}. 
\end{proof}

The following statement is an immediate consequence of Theorem~\ref{thm:interval} and Lemma~\ref{retracts-inherit}.

\begin{corollary}\label{cor:metric} Suppose $X$ is a metric space. Fix an integer $n$ such that there exists a Lipschitz retraction $r\colon X(n)\to X(n-1)$. Suppose $I\subset \mathbb R$ is a nondegenerate compact interval and there exist Lipschitz maps $f\colon X\to I$ and $g\colon I\to X$ such that $f\circ g = \id_{I}$. Then 
\[
\Lip(r) \ge \frac{n-3}{\Lip(f) \Lip(g)}.
\]
\end{corollary}

\begin{proof}[Proof of Theorem~\ref{thm:metric}] 
Suppose $r\colon X(n)\to X(n-1)$ is a Lipschitz retraction.

\textbf{Case 1: $X$ is a normed space.} Let $I = [0, 1]$. Pick a unit vector $u\in X$ and its norming functional $\varphi\in X^*$, that is, a linear functional $\varphi\colon X\to \mathbb R$ such that $\|\varphi\|_{X^*} = 1$ and $\varphi(u)=1$. The existence of such $\varphi$ follows from the Hahn-Banach theorem. Define $g\colon I\to X$ by $g(t)=tu$, and $f\colon X\to I$ by 
$f(x) = \min(\max(\varphi(x), 0), 1)$. 
Note that both $f$ and $g$ are $1$-Lipschitz and $f\circ g = \id_I$. By Corollary~\ref{cor:metric} we have $\Lip(r)\ge n-3$. 

\textbf{Case 2: $X$ is an Hadamard space.}
Pick any two distinct points $p, q\in X$ and let $I = [0, d_X(p, q)]$. Since Hadamard spaces are geodesic, there exists an isometric embedding $g\colon I\to X$, namely a reparametrized geodesic connecting $p$ to $q$. Since $g(I)$ is a closed convex subset of $X$, the nearest-point projection onto $X$ is a $1$-Lipschitz map~\cite[Theorem 2.1.12]{Bacak}. Let $f$ be the composition of this projection with $g^{-1}$. Since $\Lip(f)=1=\Lip(g)$,  Corollary~\ref{cor:metric} yields $\Lip(r)\ge n-3$ as claimed. 
\end{proof}

\bibliography{references.bib} 
\bibliographystyle{plain} 

\end{document}